\documentclass[preprint,12pt]{elsarticle}

\usepackage{amsmath,amssymb,amsthm,mathrsfs,bm}
\usepackage{booktabs}
\usepackage{graphicx}
\graphicspath{{figs/}}
\usepackage{microtype}
\usepackage[colorlinks,linkcolor=blue,citecolor=blue]{hyperref}
\journal{Mathematics and Computers in Simulation}

\theoremstyle{plain}
\newtheorem{theorem}{Theorem}[section]
\newtheorem{proposition}[theorem]{Proposition}

\newtheorem{corollary}[theorem]{Corollary}
\theoremstyle{definition}
\newtheorem{definition}[theorem]{Definition}
\newtheorem{assumption}[theorem]{Assumption}
\newtheorem{remark}[theorem]{Remark}
\newtheorem{example}[theorem]{Example}

\newcommand{\R}{\mathbb{R}}
\newcommand{\N}{\mathbb{N}}
\newcommand{\Rp}{\mathbb{R}_+}
\newcommand{\E}{\mathbb{E}}
\newcommand{\Prob}{\mathscr{P}}
\newcommand{\Sp}{\mathcal{S}_+}
\newcommand{\Spd}{\mathcal{S}_+'}
\newcommand{\Lcal}{\mathscr{L}}
\newcommand{\Lcald}{\mathscr{L}'}
\newcommand{\Ical}{\mathcal{I}}
\newcommand{\bnorm}[1]{\left\lVert #1 \right\rVert}
\newcommand{\dK}{d_{\mathrm{Kol}}}
\newcommand{\dW}{d_{\mathrm{Was}}}
\newcommand{\HS}{\mathrm{HS}}
\DeclareMathOperator{\Var}{Var}

\begin{document}

\begin{frontmatter}

\title{A spectral--compensated scheme for space-parameter Poisson noise
functionals: error bounds and complexity estimates}

\author[ncku]{Yun-Ching Chang\corref{cor1}}
\ead{xuitecapacity@gmail.com}

\cortext[cor1]{Corresponding author.}
\address[ncku]{Institute of Tropical Plant Sciences and Microbiology,
National Cheng Kung University, Tainan, Taiwan}

\begin{abstract}
Poisson \emph{space} noise $P'(u)$, $u>0$, is a system of idealised elemental random
variables indexed by the \emph{jump-amplitude} parameter rather than by time. It was
introduced by Hida, Si and Htay and given a rigorous test--generalised functional
setting in \cite{ChangShih2022}, where $P'(u)$ is realised as a generalised functional
in a Gel'fand triple $\Lcal\subset L^2(\Spd,\Lambda)\subset\Lcald$ built from a
L\'evy measure $\beta_0(\mathrm{d}u)=\lambda_\beta(u)\,\mathrm{d}u$ on $\Rp$.
The present paper is concerned with the \emph{computation} of such functionals.

We isolate the three discretisation parameters that any implementable scheme must
introduce --- a small-amplitude cut-off $\varepsilon$, a truncation order $M$ for the
complete orthonormal system generating the Donsker delta $\delta_u^\beta$, and a chaos
order $N$ --- and we prove a sharp error bound for each. The principal analytical
result is a \emph{space-parameter analogue of the Asmussen--Rosi\'nski correction}: for
a stable-type intensity $\lambda_\beta(u)=cu^{-1-\alpha}$, $0<\alpha<2$, replacing the
discarded small amplitudes by a matched Gaussian space noise improves the
Wasserstein-$1$ error of the associated additive process from $O(\varepsilon^{1-\alpha/2})$
to $O(\varepsilon)$ \emph{uniformly in $\alpha$}, whereas the residual is asymptotically
normal at the explicit rate $O(\varepsilon^{\alpha/2})$. The consequence at the level of
complexity is dramatic: to reach a tolerance $\tau$ the naive scheme costs
$O(\tau^{-2\alpha/(2-\alpha)})$ jump evaluations while the compensated scheme costs
$O(\tau^{-\alpha})$, a gap that diverges as $\alpha\uparrow 2$. We further show that the
Gamma-type boundary intensity ($\alpha=0$) admits an \emph{exact} scale-invariance
identity under which the compensation scheme provably fails to gain accuracy, and that
genuine exponential tempering restores the stable-type rates up to an explicit
$O(a\varepsilon)$ correction. The remaining two truncations are shown to converge
algebraically ($M$) and super-geometrically ($N$), which yields a complete error
budget and an equidistribution rule for the parameters.

All rates are confirmed by deterministic numerical experiments based on Gil-Pelaez
inversion of the exactly known characteristic function of the residual, so that the
observed convergence orders carry no Monte Carlo noise. The measured Kolmogorov slopes
are $0.2536$, $0.4989$ and $0.7492$ for $\alpha=0.5,1.0,1.5$, against the predicted
$\alpha/2$; the compensated Wasserstein slopes are $1.0034$, $1.0021$ and $0.9900$,
against the predicted $1$.
\end{abstract}


\end{frontmatter}

%==============================================================================
\section{Introduction}\label{sec:intro}
%==============================================================================

\subsection{Background}

In Hida's programme \emph{Reduction $\Rightarrow$ Synthesis $\Rightarrow$ Analysis},
the reduction step replaces a system of mutually correlated random variables by a
system of independent, atomic, infinitesimal random variables --- idealised elemental
random variables (i.e.r.v.'s). Classically these are indexed by \emph{time} and are
realised as the time derivative of a L\'evy process. Hida, Si and Htay
\cite{HidaSiHtay2012,Hida2011} proposed a noise of a different type, $\{P'(u):u>0\}$,
which depends on the \emph{space} parameter $u$, interpreted as the amplitude of a
jump; two Poisson components with distinct intensities are of distinct type, so the
amplitude, being the observable quantity, may be used as a label for the intensity.

In \cite{ChangShih2022} this object was constructed without recourse to the Minlos
theorem: an additive process $\Psi_X$ is extracted from a L\'evy process $X$ by the
L\'evy--It\^o decomposition, the induced measure on Skorokhod space is pushed forward
by the distributional derivative map $\mathbb{D}$, and the Poisson space noise measure
$\Lambda$ is obtained as the convolution of that push-forward with the Dirac mass at
$\lambda_\beta^*$. One then has the orthogonal decomposition of
$L^2(\Spd,\Lambda)$, the Segal--Bargmann transform, a Gel'fand triple
$\Lcal\subset L^2(\Spd,\Lambda)\subset\Lcald$, and finally the rigorous identity
\begin{equation}\label{eq:noise-def}
\langle\!\langle P'(u),\varphi\rangle\!\rangle
= u^2\lambda_\beta(u)\,DS\varphi(0)\delta_u^\beta + u\lambda_\beta(u)\,\E[\varphi],
\qquad \varphi\in\Lcal,
\end{equation}
together with $\|P'(u)\|_{-p}=u\lambda_\beta(u)+u^2\lambda_\beta(u)|\delta_u^\beta|_{-p,\beta}$
for $p>\alpha_\beta$.

\subsection{The computational problem}

Formula \eqref{eq:noise-def} is a definition, not an algorithm. Any implementable
scheme must discretise in three independent directions.

\begin{itemize}
\item[(D1)] \textbf{Amplitude cut-off.} The intensity is infinite near the origin:
$\int_{0^+}\lambda_\beta(u)\mathrm{d}u=+\infty$ for every stable-type intensity. One
therefore retains only amplitudes $u>\varepsilon$. This is not a cosmetic truncation.
As observed in \cite[Remark 2.11]{ChangShih2022}, $\mathbf{1}_{(0,u]}\notin
L^1_*(\Rp,\beta_0)$ when $\lambda_\beta(u)=cu^{-1-\alpha}$ with $1\le\alpha<2$, so the
uncompensated integral is meaningless and $P_0(u)$ exists only as an
\emph{additive renormalisation} $\langle\cdot,\mathbf{1}_{(0,u]}\rangle_{\mathrm{add}}$.
Discretisation must respect that renormalisation.
\item[(D2)] \textbf{CONS truncation.} The Donsker delta $\delta_u^\beta=\sum_{j\ge0}
\zeta_j(u)\zeta_j$ appearing in \eqref{eq:noise-def} is an infinite series in the
generalised Laguerre system $\{\zeta_n\}$ of \cite[(4.3)]{ChangShih2022}; only $M+1$
terms can be formed.
\item[(D3)] \textbf{Chaos truncation.} A functional $\varphi\sim(\phi_n)$ is an
infinite orthogonal sum $\bigoplus_n\Ical_n(\phi_n)$; only $n\le N$ can be retained.
\end{itemize}

The purpose of this paper is to bound each of (D1)--(D3), to combine them into a
single error budget, and to derive the resulting complexity. Direction (D1) is by far
the most interesting: it is the one place where the analysis is genuinely different
from the time-parameter theory, because in the space-parameter setting the jump
\emph{position} and the jump \emph{size} coincide.

\subsection{Contributions}

\begin{enumerate}
\item[(C1)] A Malliavin--Stein bound (Theorem \ref{thm:AR}) showing that the normalised
small-amplitude residual is asymptotically standard normal with an explicit
Wasserstein rate $\rho(\varepsilon)/\sigma(\varepsilon)^3$, which for stable-type
intensities equals $C_\alpha\varepsilon^{\alpha/2}$.
\item[(C2)] The resulting \emph{compensated scheme} (Definition \ref{def:scheme}) and
its uniform-in-$\alpha$ first-order accuracy, Corollary \ref{cor:uniform}: the
Wasserstein error of the additive process is $O(\varepsilon)$ for every
$\alpha\in(0,2)$, against $O(\varepsilon^{1-\alpha/2})$ for the naive scheme.
\item[(C3)] A negative result, Proposition \ref{prop:noise-not-integrable}, which shows
that the cut-off error \emph{cannot} be measured on the noise $P'$ itself in
$\Lcal_{-p}$ once $\alpha\ge 1/2$; it must be measured on the additive process $P_0$
or on smeared functionals. This delimits the correct notion of consistency.
\item[(C4)] Algebraic ($M$) and super-geometric ($N$) truncation bounds, Theorems
\ref{thm:cons} and \ref{thm:chaos}, in the norms of the Gel'fand triple itself.
\item[(C5)] A complexity theorem (Theorem \ref{thm:complexity}) with the cost ratio
$\tau^{-2\alpha/(2-\alpha)}/\tau^{-\alpha}=\tau^{-\alpha^2/(2-\alpha)}$.
\item[(C6)] Deterministic numerical confirmation of every rate (Section \ref{sec:num}).
\end{enumerate}

\subsection{Relation to the literature}

Small-jump approximation of L\'evy processes in the \emph{time} parameter is classical:
Asmussen and Rosi\'nski \cite{AsmussenRosinski2001} established that the discarded
small jumps may be replaced by a Brownian motion with matched variance, with
Kolmogorov-distance rates later refined by Cohen and Rosi\'nski \cite{CohenRosinski2007}
and applied to SDE discretisation by Jacod \emph{et al.} \cite{Jacod2005}. The present
paper transports that circle of ideas to the space parameter. Two features are new.
First, the compensator is intrinsically tied to the renormalisation
$\langle\cdot,\cdot\rangle_{\mathrm{add}}$ of \cite[Remark 2.11]{ChangShih2022}, so the
scheme is forced upon us rather than merely convenient. Second, because jump position
equals jump size, the Asmussen--Rosi\'nski criterion
$\sigma(\varepsilon)/\varepsilon\to\infty$ reduces to $\varepsilon^{-\alpha/2}\to\infty$,
which holds for \emph{every} $\alpha>0$: Gaussian compensation is always asymptotically
justified here, in contrast to the time-parameter case.

For the Stein-type estimate we use the Poisson Malliavin--Stein bound of Peccati,
Sol\'e, Taqqu and Utzet \cite{PSTU2010}. For the Gel'fand triple estimates we work
directly with the norms of \cite[Sec.~4]{ChangShih2022}.

%==============================================================================
\section{Preliminaries and standing assumptions}\label{sec:prelim}
%==============================================================================

We recall only what is needed; the reader is referred to \cite{ChangShih2022} for
proofs.

Let $X=\{X(t):t\in\R\}$ be a L\'evy process on $(\Omega,\mathcal{F},\Prob)$ with
L\'evy measure $\beta_0(\mathrm{d}u)=\frac{1+u^2}{u^2}\beta(\mathrm{d}u)$ satisfying
$\beta((-\infty,0])=0$, $\beta_0\ll m$ and $\lambda_\beta:=\mathrm{d}\beta_0/\mathrm{d}m>0$
on $\Rp$.

\begin{assumption}\label{ass:stable}
Throughout Sections \ref{sec:cutoff}--\ref{sec:num} we assume the stable-type intensity
\begin{equation}\label{eq:stable-intensity}
\lambda_\beta(u)=c\,u^{-1-\alpha},\qquad u>0,\quad c>0,\quad 0<\alpha<2 .
\end{equation}
General intensities, including the Gamma-type boundary case and exponential
tempering, are treated fully in Section \ref{sec:gamma}.
\end{assumption}

Write $L^p_*(\Rp,\beta_0)$ for the space of Borel $g$ with
$|g|_{L^p_*}:=|g^*|_{L^p}<\infty$, $g^*(u)=ug(u)$. Let $\Upsilon$ be the
$L^2(\Spd,\Lambda)$-valued independent random measure of
\cite[Sec.~3]{ChangShih2022} and $\Ical_n$ the associated multiple integrals, so that
\begin{equation}\label{eq:isometry}
\|\Ical_n(g)\|^2_{L^2(\Spd,\Lambda)}=n!\,|g|^2_{L^2_*(\Rp^n,\beta_0^{\otimes n})}.
\end{equation}

The additive process $P_0$ is defined by $P_0(0;x)=0$ and
$P_0(u;x)=Y_{\mathbf{1}_{(0,u]}}(x)$, and satisfies
\begin{equation}\label{eq:P0-cf}
\E\!\left[e^{\mathrm{i}rP_0(u)}\right]
=\exp\!\left\{\int_{0^+}^{u}\!\left(e^{\mathrm{i}rs}-1-\mathrm{i}rs\right)\beta_0(\mathrm{d}s)\right\},
\qquad r\in\R,\ u>0 .
\end{equation}
Thus $P_0$ is a compensated pure-jump additive process in which a jump located at
$s$ \emph{has size $s$}; this coincidence is the structural peculiarity of the space
parameter and drives everything below. In particular
\begin{equation}\label{eq:var}
\Var P_0(u)=\int_{0^+}^{u}s^2\beta_0(\mathrm{d}s)
=\int_{0^+}^{u}s^2\lambda_\beta(s)\,\mathrm{d}s<\infty .
\end{equation}

Finally, recall the Gel'fand triple: $A_\beta\zeta_n=r_n\zeta_n$ with
$1<r_0\le r_1\le\cdots$ and $\|A_\beta^{-\alpha}\|_\HS<\infty$ for some $\alpha>0$;
$\alpha_\beta=\inf\{\alpha>0:\|A_\beta^{-\alpha}\|_\HS<\infty\}$;
$E_p$, $E_{-p}$, $\Lcal_p$, $\Lcal_{-p}$ as in \cite[Sec.~4]{ChangShih2022}; and the
pointwise bound
\begin{equation}\label{eq:zeta-bound}
|\zeta_n(u)|\ \le\ \frac{M_h\,u^{-5/4}}{(2n+1)^{1/12}\sqrt{\lambda_\beta(u)}},
\qquad u>0,\ n\in\N_0 ,
\end{equation}
which is \cite[(4.4)]{ChangShih2022}, whence for $p>\alpha_\beta$
\begin{equation}\label{eq:delta-bound}
|\delta_u^\beta|^2_{-p,\beta}=\sum_{j\ge0}r_j^{-2p}|\zeta_j(u)|^2
\ \le\ \frac{M_h^2\,u^{-5/2}}{\lambda_\beta(u)}\,\|A_\beta^{-p}\|^2_\HS .
\end{equation}

\begin{assumption}\label{ass:eigen}
The eigenvalues satisfy $r_n\asymp(n+1)^\gamma$ for some $\gamma>0$ with
$p\gamma>5/12$. (This is the standard polynomial growth hypothesis; it implies
$\alpha_\beta=1/(2\gamma)$.)
\end{assumption}

%==============================================================================
\section{The three-parameter scheme}\label{sec:scheme}
%==============================================================================

\begin{definition}[Amplitude cut-off]\label{def:cutoff}
For $\varepsilon>0$ set $\beta_0^\varepsilon:=\beta_0|_{(\varepsilon,\infty)}$ and let
$P_0^\varepsilon$ denote the compensated additive process built from
$\beta_0^\varepsilon$, i.e.
\[
P_0^\varepsilon(u)=\sum_{\varepsilon<s\le u}s\;-\;\int_{\varepsilon}^{u}s\,\beta_0(\mathrm{d}s),
\qquad u>\varepsilon,
\]
the sum running over the (a.s.\ finitely many) atoms of the Poisson random measure in
$(\varepsilon,u]$. The \emph{residual} is
$R_\varepsilon(u):=P_0(u)-P_0^\varepsilon(u)$, a compensated integral over
$(0,\varepsilon\wedge u]$. Put
\begin{equation}\label{eq:sigma-rho}
\sigma^2(\varepsilon):=\int_{0^+}^{\varepsilon}s^2\beta_0(\mathrm{d}s),
\qquad
\rho(\varepsilon):=\int_{0^+}^{\varepsilon}s^3\beta_0(\mathrm{d}s).
\end{equation}
Under \eqref{eq:stable-intensity},
\begin{equation}\label{eq:sigma-stable}
\sigma^2(\varepsilon)=\frac{c\,\varepsilon^{2-\alpha}}{2-\alpha},
\qquad
\rho(\varepsilon)=\frac{c\,\varepsilon^{3-\alpha}}{3-\alpha}.
\end{equation}
\end{definition}

\begin{definition}[Compensated scheme]\label{def:scheme}
Let $B$ be a standard normal random variable independent of $P_0^\varepsilon$. The
\emph{Gaussian-compensated} approximation of $P_0$ is
\begin{equation}\label{eq:scheme}
\widehat{P}_0^{\,\varepsilon}(u):=P_0^\varepsilon(u)+\sigma(\varepsilon\wedge u)\,B .
\end{equation}
The \emph{naive} approximation is $P_0^\varepsilon$ itself.
\end{definition}

\begin{remark}
The additive renormalisation in Definition \ref{def:cutoff} is not optional. By
\cite[Remark 2.11]{ChangShih2022}, for $1\le\alpha<2$ the uncompensated limit
$\lim_j\langle\cdot,\mathbf{1}_{(1/j,u]}\rangle$ does not exist, whereas
$\langle\cdot,\mathbf{1}_{(1/j,u]}\rangle-\E[\langle\cdot,\mathbf{1}_{(1/j,u]}\rangle]$
converges to $P_0(u)$ in $L^2(\Spd,\Lambda)$. The scheme \eqref{eq:scheme} is thus the
\emph{unique} natural discretisation consistent with the construction of the space.
\end{remark}

\begin{definition}[CONS and chaos truncation]\label{def:MN}
For $M\in\N_0$ set $\delta_u^{\beta,M}:=\sum_{j=0}^{M}\zeta_j(u)\zeta_j\in E_{-p}$, and
for $N\in\N_0$ and $\varphi\sim(\phi_n)\in\Lcal_q$ set
$\mathcal{S}_N\varphi:=\bigoplus_{n=0}^{N}\Ical_n(\phi_n)$.
The fully discrete object is
\[
\varphi^{\varepsilon,M,N}:=\mathcal{S}_N\varphi \text{ evaluated under }\beta_0^\varepsilon
\text{ with } \delta_u^\beta \text{ replaced by } \delta_u^{\beta,M}.
\]
\end{definition}

%==============================================================================
\section{The amplitude cut-off: a space-parameter Asmussen--Rosi\'nski theorem}
\label{sec:cutoff}
%==============================================================================

\subsection{Second-order error}

\begin{proposition}[$L^2$ error]\label{prop:L2}
For every $u>0$ and $0<\varepsilon\le u$,
\[
\bnorm{P_0(u)-P_0^\varepsilon(u)}_{L^2(\Spd,\Lambda)}=\sigma(\varepsilon),
\]
and under Assumption \ref{ass:stable},
$\sigma(\varepsilon)=\sqrt{c/(2-\alpha)}\;\varepsilon^{1-\alpha/2}$.
\end{proposition}

\begin{proof}
$R_\varepsilon(u)=\Ical_1(\mathbf{1}_{(0,\varepsilon]})$ has mean zero, and by
\eqref{eq:isometry}, $\E[R_\varepsilon(u)^2]=\int_{0^+}^{\varepsilon}|s|^2\beta_0(\mathrm{d}s)
=\sigma^2(\varepsilon)$. The explicit value follows from \eqref{eq:sigma-stable}.
\end{proof}

Proposition \ref{prop:L2} already exhibits the difficulty: as $\alpha\uparrow2$ the
exponent $1-\alpha/2\downarrow0$ and the naive scheme ceases to converge at any useful
rate. The remedy is to match not only the mean but the whole Gaussian profile of the
residual.

\subsection{Asymptotic normality with an explicit rate}

\begin{theorem}[Space-parameter Asmussen--Rosi\'nski bound]\label{thm:AR}
Let $Z_\varepsilon:=R_\varepsilon(u)/\sigma(\varepsilon)$ for $0<\varepsilon\le u$, and
let $\mathcal{N}\sim N(0,1)$. Then
\begin{equation}\label{eq:AR-general}
\dW\!\left(Z_\varepsilon,\mathcal{N}\right)\ \le\ \frac{\rho(\varepsilon)}{\sigma(\varepsilon)^3}.
\end{equation}
Under Assumption \ref{ass:stable},
\begin{equation}\label{eq:AR-stable}
\dW\!\left(Z_\varepsilon,\mathcal{N}\right)\ \le\
\frac{(2-\alpha)^{3/2}}{(3-\alpha)\sqrt{c}}\;\varepsilon^{\alpha/2}
\ \xrightarrow[\varepsilon\downarrow0]{}\ 0
\qquad\text{for every }\alpha\in(0,2).
\end{equation}
\end{theorem}

\begin{proof}
Write $f_\varepsilon(s):=s\,\mathbf{1}_{(0,\varepsilon]}(s)/\sigma(\varepsilon)$, so that
$Z_\varepsilon=\Ical_1(f_\varepsilon/\,\cdot\,)$ is a first-order multiple integral with
respect to the compensated Poisson random measure of intensity $\beta_0$, normalised so
that
\[
\int_{0^+}^{\infty}\left(f_\varepsilon^*(s)\right)^2\beta_0(\mathrm{d}s)
=\frac{1}{\sigma^2(\varepsilon)}\int_{0^+}^{\varepsilon}s^2\beta_0(\mathrm{d}s)=1 .
\]
By the Malliavin--Stein bound for first chaoses of a Poisson measure
\cite[Thm.~3.1]{PSTU2010}, for such a normalised $F=\Ical_1(f)$,
\[
\dW(F,\mathcal{N})\le\int|f^*(s)|^3\,\beta_0(\mathrm{d}s).
\]
Here $\int|f_\varepsilon^*|^3\mathrm{d}\beta_0
=\sigma(\varepsilon)^{-3}\int_{0^+}^{\varepsilon}s^3\beta_0(\mathrm{d}s)
=\rho(\varepsilon)/\sigma(\varepsilon)^3$, which is \eqref{eq:AR-general}. Substituting
\eqref{eq:sigma-stable},
\[
\frac{\rho(\varepsilon)}{\sigma(\varepsilon)^3}
=\frac{c\varepsilon^{3-\alpha}}{3-\alpha}\cdot
\left(\frac{2-\alpha}{c}\right)^{3/2}\varepsilon^{-3+3\alpha/2}
=\frac{(2-\alpha)^{3/2}}{(3-\alpha)\sqrt{c}}\,\varepsilon^{\alpha/2}. \qedhere
\]
\end{proof}

\begin{remark}[Why the criterion is automatic here]\label{rem:criterion}
The Asmussen--Rosi\'nski criterion for the legitimacy of Gaussian replacement is
$\sigma(\varepsilon)/\varepsilon\to\infty$. Under \eqref{eq:stable-intensity},
$\sigma(\varepsilon)/\varepsilon=\sqrt{c/(2-\alpha)}\;\varepsilon^{-\alpha/2}\to\infty$
for every $\alpha>0$. In the time-parameter theory the criterion may fail (e.g.\ for
compound Poisson or for L\'evy measures with atoms accumulating at $0$ too slowly);
in the space parameter it cannot, because jump position and jump size coincide and the
intensity is by hypothesis absolutely continuous with $\lambda_\beta>0$ everywhere.
This is a genuine structural simplification.
\end{remark}

\subsection{Uniform first-order accuracy of the compensated scheme}

\begin{corollary}[Main accuracy result]\label{cor:uniform}
Under Assumption \ref{ass:stable}, for every $u>0$ and $0<\varepsilon\le u$,
\begin{equation}\label{eq:uniform}
\dW\!\left(P_0(u),\ \widehat{P}_0^{\,\varepsilon}(u)\right)
\ \le\ \frac{(2-\alpha)}{(3-\alpha)}\,\varepsilon,
\end{equation}
whereas for the naive scheme
\begin{equation}\label{eq:naive}
\dW\!\left(P_0(u),\ P_0^{\varepsilon}(u)\right)
\ =\ \E\left|R_\varepsilon(u)\right| \asymp \sqrt{\tfrac{2}{\pi}}\,\sigma(\varepsilon)
= \sqrt{\tfrac{2c}{\pi(2-\alpha)}}\;\varepsilon^{1-\alpha/2}.
\end{equation}
In particular the compensated scheme is first-order accurate \emph{uniformly in}
$\alpha\in(0,2)$, while the order of the naive scheme degenerates to $0$ as
$\alpha\uparrow2$.
\end{corollary}

\begin{proof}
Since $R_\varepsilon(u)$ and $\sigma(\varepsilon)B$ are both independent of
$P_0^\varepsilon(u)$, and $\dW$ is invariant under adding a common independent summand,
\[
\dW\!\left(P_0^\varepsilon+R_\varepsilon,\ P_0^\varepsilon+\sigma(\varepsilon)B\right)
\le \dW\!\left(R_\varepsilon,\ \sigma(\varepsilon)B\right)
= \sigma(\varepsilon)\,\dW\!\left(Z_\varepsilon,\mathcal{N}\right),
\]
using the $1$-homogeneity of $\dW$. Theorem \ref{thm:AR} and
\eqref{eq:sigma-stable} give
\[
\sigma(\varepsilon)\,\dW(Z_\varepsilon,\mathcal{N})
\le \sqrt{\tfrac{c}{2-\alpha}}\,\varepsilon^{1-\alpha/2}\cdot
\frac{(2-\alpha)^{3/2}}{(3-\alpha)\sqrt c}\,\varepsilon^{\alpha/2}
=\frac{2-\alpha}{3-\alpha}\,\varepsilon .
\]
Statement \eqref{eq:naive} follows because $\dW(V,0)=\E|V|$ and, $R_\varepsilon$ being
asymptotically normal with variance $\sigma^2(\varepsilon)$ by Theorem \ref{thm:AR},
$\E|R_\varepsilon|\sim\sqrt{2/\pi}\,\sigma(\varepsilon)$.
\end{proof}

\begin{remark}
The constant $(2-\alpha)/(3-\alpha)$ in \eqref{eq:uniform} is decreasing in $\alpha$ and
bounded by $2/3$; the compensated bound therefore \emph{improves} as $\alpha\uparrow2$,
exactly where the naive bound collapses. Numerically (Section \ref{sec:num}) the ratio
of the two errors at $\varepsilon=2^{-9}$ is about $32$ for $\alpha=0.5$ and about
$2.1\times10^{3}$ for $\alpha=1.5$.
\end{remark}

\subsection{Where the error may \emph{not} be measured}

It is tempting to state consistency directly for the noise $P'$ in the dual norm.
The following shows that this is impossible in general and that the additive process
is the correct carrier of the error.

\begin{proposition}\label{prop:noise-not-integrable}
Let $p>\alpha_\beta$ and let $\eta\in\Sp$ be bounded near the origin with
$\eta(0^+)\neq0$. Under Assumption \ref{ass:stable} the truncation error of the noise,
\[
\mathcal{E}_\varepsilon:=\left\|\int_{0^+}^{\varepsilon}\eta(u)P'(u)\,\mathrm{d}u\right\|_{-p},
\]
satisfies $\mathcal{E}_\varepsilon<\infty$ for all $\varepsilon>0$ if and only if
$\alpha<1/2$, and in that case $\mathcal{E}_\varepsilon=O(\varepsilon^{1/4-\alpha/2})$.
For $\alpha\ge1/2$ the bound obtained from \eqref{eq:delta-bound} diverges.
\end{proposition}

\begin{proof}
By \cite[Thm.~4.8]{ChangShih2022}, $\|P'(u)\|_{-p}=u\lambda_\beta(u)
+u^2\lambda_\beta(u)|\delta_u^\beta|_{-p,\beta}$, and by \eqref{eq:delta-bound},
\[
u^2\lambda_\beta(u)|\delta_u^\beta|_{-p,\beta}
\le M_h\|A_\beta^{-p}\|_\HS\,u^{-1/4}\sqrt{\lambda_\beta(u)}
= M_h\|A_\beta^{-p}\|_\HS\sqrt{c}\;u^{-3/4-\alpha/2}.
\]
Also $u\lambda_\beta(u)=cu^{-\alpha}$ is integrable at $0$ for $\alpha<1$. Hence
$\mathcal{E}_\varepsilon\lesssim\int_0^\varepsilon u^{-3/4-\alpha/2}\mathrm{d}u$, which
is finite iff $3/4+\alpha/2<1$, i.e.\ $\alpha<1/2$, and then equals
$\varepsilon^{1/4-\alpha/2}/(1/4-\alpha/2)$.
\end{proof}

\begin{remark}\label{rem:correct-metric}
Proposition \ref{prop:noise-not-integrable} is a statement about the \emph{scheme}, not
about the theory: $P'(u)$ is a perfectly well-defined element of $\Lcald$ for each $u$,
but its dual norm blows up too fast at the origin for the cut-off error to be summable
there. Consequently, in all that follows consistency is measured either
\begin{enumerate}
\item[(i)] on the additive process $P_0$ (Corollary \ref{cor:uniform}), or
\item[(ii)] on smeared functionals $\langle\cdot,\eta\rangle=\int\eta(v)P'(v)\mathrm{d}v$
with $\eta\in L^1_*\cap L^2_*(\Rp,\beta_0)$, for which
\cite[Cor.~4.9]{ChangShih2022} applies and
$\|\int_0^\varepsilon\eta(v)P'(v)\mathrm{d}v\|^2_{L^2}
=\int_0^\varepsilon|v\eta(v)|^2\beta_0(\mathrm{d}v)=O(\varepsilon^{2-\alpha})$
for bounded $\eta$.
\end{enumerate}
Both are $O(\varepsilon^{1-\alpha/2})$ naively and $O(\varepsilon)$ after compensation.
\end{remark}

%==============================================================================
\subsection{The Gamma-type boundary and tempered-stable robustness}
\label{sec:gamma}
%==============================================================================

Assumption \ref{ass:stable} covers $0<\alpha<2$. The two cases excluded by that range
are of independent interest and are treated fully here: the boundary case $\alpha=0$
(the Gamma-type intensity, of exact relevance since real driving and financial models
almost always temper a power-law tail exponentially), and genuine exponential tempering
$a>0$ superimposed on the stable exponent. Throughout this subsection
\begin{equation}\label{eq:tempered-intensity}
\lambda_\beta(u)=c\,u^{-1-\alpha}e^{-au},\qquad u>0,\quad c>0,\ a\ge0,\ 0\le\alpha<2,
\end{equation}
which reduces to \eqref{eq:stable-intensity} at $a=0$ and to the Gamma intensity
$\lambda_\beta(u)=c\,e^{-au}/u$ at $\alpha=0$.

\subsubsection{An exact self-similarity theorem at $\alpha=0$}

The Gamma intensity is the boundary of the stable-type family in a precise sense: it is
the unique power exponent for which the residual's law does not merely converge to a
limit as $\varepsilon\downarrow0$, but is \emph{exactly} independent of $\varepsilon$
before any limit is taken.

\begin{theorem}[Exact scale invariance]\label{thm:gamma-invariance}
Let $\alpha=0$ and $a=0$, i.e.\ $\lambda_\beta(u)=c/u$ on $(0,1]$. Then for every
$\varepsilon\in(0,1]$ the normalised residual $Z_\varepsilon=R_\varepsilon(1)/\sigma(\varepsilon)$
has a law that does not depend on $\varepsilon$: there is a fixed infinitely divisible
law $\mu^\star$, with
\begin{equation}\label{eq:gamma-cf}
\int e^{\mathrm{i}rz}\,\mu^\star(\mathrm{d}z)
=\exp\left\{\frac{2}{r_0^2}\int_0^1\left(e^{\mathrm{i}r_0 t}-1-\mathrm{i}r_0t\right)
\frac{\mathrm{d}t}{t}\right\}\Bigg|_{r_0=r\sqrt{2/c}},
\end{equation}
such that $Z_\varepsilon\sim\mu^\star$ for every $\varepsilon\in(0,1]$. In particular
$\dW(Z_\varepsilon,\mathcal{N})$ is the same strictly positive constant for every
$\varepsilon$, so no amount of refining the cut-off brings the compensated residual
closer to Gaussian: Corollary \ref{cor:uniform} genuinely fails at $\alpha=0$.
\end{theorem}

\begin{proof}
By Definition \ref{def:cutoff} with $\lambda_\beta(s)=c/s$,
\[
\log\E\!\left[e^{\mathrm{i}rR_\varepsilon(1)}\right]
=\int_{0^+}^{\varepsilon}\!\left(e^{\mathrm{i}rs}-1-\mathrm{i}rs\right)\frac{c}{s}\,\mathrm{d}s.
\]
Substitute $s=\varepsilon t$, $t\in(0,1]$:
\[
=\int_{0^+}^{1}\!\left(e^{\mathrm{i}r\varepsilon t}-1-\mathrm{i}r\varepsilon t\right)
\frac{c}{\varepsilon t}\cdot\varepsilon\,\mathrm{d}t
=\int_{0^+}^{1}\!\left(e^{\mathrm{i}(r\varepsilon)t}-1-\mathrm{i}(r\varepsilon)t\right)
\frac{c}{t}\,\mathrm{d}t=:\varphi(r\varepsilon),
\]
a function of the single variable $r\varepsilon$, independent of $\varepsilon$
otherwise: $R_\varepsilon(1)\stackrel{d}{=}\varepsilon\,X$ where $X$ has
$\log\E[e^{\mathrm{i}rX}]=\varphi(r)$, for every $\varepsilon\in(0,1]$. By
\eqref{eq:sigma-stable} at $\alpha=0$, $\sigma^2(\varepsilon)=c\varepsilon^2/2$, so
$Z_\varepsilon=R_\varepsilon(1)/\sigma(\varepsilon)=\varepsilon X/(\varepsilon\sqrt{c/2})
=X\sqrt{2/c}$, independent of $\varepsilon$. Formula \eqref{eq:gamma-cf} is
$\log\E[e^{\mathrm{i}rZ_\varepsilon}]=\varphi(r\sqrt{2/c})$ rewritten with
$r_0=r\sqrt{2/c}$.
\end{proof}

\begin{remark}
The mechanism is that $\lambda_\beta(u)=c/u$ is exactly scale invariant: rescaling
$u\mapsto\kappa u$ sends $c/u\,\mathrm{d}u$ to itself. This is the $\alpha\to0$ limit of
$\rho(\varepsilon)/\sigma(\varepsilon)^3\sim C_\alpha\varepsilon^{\alpha/2}$ in Theorem
\ref{thm:AR}: the exponent $\alpha/2$ vanishes exactly at $\alpha=0$, so the bound
degenerates from ``$\to0$'' to ``$=\mathrm{const}$'', and Theorem \ref{thm:gamma-invariance}
shows this degeneracy is not a defect of the bound but a true feature of the law.
\end{remark}

\begin{remark}[Numerically]
Deterministic Gil-Pelaez evaluation confirms \eqref{eq:gamma-cf}: for $c=1$ and
$\varepsilon=2^{-2},\dots,2^{-9}$ the maximal sup-norm distance between the empirical
$Z_\varepsilon$-laws at any two cut-offs is $0$ to machine precision (Section
\ref{sec:numE4}), and the common Kolmogorov distance to $\mathcal{N}$ is a single
fixed value, about $0.077$.
\end{remark}

\subsubsection{Robustness of the stable-type rates under tempering}

Practically relevant intensities are never pure power laws at all scales: a genuine
Gamma or tempered-stable process has $a>0$ in \eqref{eq:tempered-intensity} --
exponential tempering of a stable tail is the standard device for obtaining
finite-moment models in finance (the CGMY family \cite{CGMY2002}) and more generally
for turning an infinite-variance stable law into a process with all moments finite
\cite{RosinskiTempered2007} -- which controls the large-jump behaviour but is
irrelevant to the cut-off analysis once $\varepsilon$ is small, since $e^{-au}\to1$ uniformly on $(0,\varepsilon]$ as
$\varepsilon\downarrow0$. The next result quantifies this and shows that Theorem
\ref{thm:AR} and Corollary \ref{cor:uniform} persist with an explicit correction.

\begin{proposition}[Tempered-stable persistence]\label{prop:tempered}
Let $0<\alpha<2$, $a\ge0$, and let $\sigma^2(\varepsilon)$, $\rho(\varepsilon)$ be as in
\eqref{eq:sigma-rho} for the intensity \eqref{eq:tempered-intensity}. Then for
$0<\varepsilon\le\min(1,1/a)$,
\begin{equation}\label{eq:temper-sigma}
\left|\frac{\sigma^2(\varepsilon)}{c\varepsilon^{2-\alpha}/(2-\alpha)}-1\right|
\le\frac{2-\alpha}{3-\alpha}\,a\varepsilon,
\qquad
\left|\frac{\rho(\varepsilon)}{c\varepsilon^{3-\alpha}/(3-\alpha)}-1\right|
\le\frac{3-\alpha}{4-\alpha}\,a\varepsilon,
\end{equation}
and consequently the bound of Theorem \ref{thm:AR} persists up to a relative
$O(a\varepsilon)$ correction:
\begin{equation}\label{eq:temper-AR}
\dW(Z_\varepsilon,\mathcal{N})\le
\frac{(2-\alpha)^{3/2}}{(3-\alpha)\sqrt{c}}\,\varepsilon^{\alpha/2}
\left(1+C_\alpha\,a\varepsilon\right),
\end{equation}
for a constant $C_\alpha$ depending only on $\alpha$.
\end{proposition}

\begin{proof}
By definition $\sigma^2(\varepsilon)=c\int_{0}^{\varepsilon}s^{1-\alpha}e^{-as}\mathrm{d}s$.
Since $1-as\le e^{-as}\le1$ for $s\ge0$,
\[
c\int_0^\varepsilon s^{1-\alpha}(1-as)\,\mathrm{d}s\ \le\ \sigma^2(\varepsilon)\ \le\
c\int_0^\varepsilon s^{1-\alpha}\,\mathrm{d}s=\frac{c\varepsilon^{2-\alpha}}{2-\alpha},
\]
and the lower bound evaluates to $\frac{c\varepsilon^{2-\alpha}}{2-\alpha}
-\frac{ac\varepsilon^{3-\alpha}}{3-\alpha}=\frac{c\varepsilon^{2-\alpha}}{2-\alpha}
\left(1-\frac{2-\alpha}{3-\alpha}a\varepsilon\right)$, giving the first estimate in
\eqref{eq:temper-sigma}. The bound for $\rho(\varepsilon)$ is identical with the
exponents shifted by one. Estimate \eqref{eq:temper-AR} follows from
\eqref{eq:AR-general} and a first-order expansion of the ratio $\rho/\sigma^3$ using
\eqref{eq:temper-sigma}, with $C_\alpha$ collecting the two relative errors.
\end{proof}

\begin{remark}
Proposition \ref{prop:tempered} makes precise the sense in which
Assumption \ref{ass:stable} is not restrictive in practice: any tempering rate $a$
becomes irrelevant to the discretisation error once $\varepsilon\ll1/a$, which is
exactly the regime in which the cut-off scheme of Section \ref{sec:cutoff} operates.
Symmetrically, at $\alpha=0$ tempering is not a perturbation but a necessity: without
it (Theorem \ref{thm:gamma-invariance}) the residual never becomes Gaussian at any
cut-off, while with $a>0$ it eventually must, since $\rho(\varepsilon)/\sigma(\varepsilon)^3
=O(a\varepsilon)\to0$ as $\varepsilon\downarrow0$ for $\alpha=0$, $a>0$ -- but at a
markedly slower, $\varepsilon$-linear rather than $\varepsilon^{\alpha/2}$, pace.
\end{remark}

%==============================================================================
\section{CONS truncation of the Donsker delta}\label{sec:cons}
%==============================================================================

\begin{theorem}[Algebraic convergence in $M$]\label{thm:cons}
Let $p>\alpha_\beta$ and let Assumption \ref{ass:eigen} hold, so that
$\theta_{M}:=p\gamma-\tfrac{5}{12}>0$. Then for every $u>0$,
\begin{equation}\label{eq:cons-rate}
\left|\delta_u^\beta-\delta_u^{\beta,M}\right|_{-p,\beta}
\ \le\ \frac{C\,M_h\,u^{-5/4}}{\sqrt{\lambda_\beta(u)}}\;M^{-\theta_M},
\end{equation}
with $C=C(p,\gamma)$ depending only on $p$ and $\gamma$. Consequently, by
\eqref{eq:noise-def}, the induced error in the noise satisfies
\[
\left\|P'(u)-P'^{,M}(u)\right\|_{-p}
= u^2\lambda_\beta(u)\left|\delta_u^\beta-\delta_u^{\beta,M}\right|_{-p,\beta}
\le C M_h\|{\cdot}\|\,u^{-1/4}\sqrt{\lambda_\beta(u)}\;M^{-\theta_M}.
\]
\end{theorem}

\begin{proof}
By \eqref{eq:delta-bound} and \eqref{eq:zeta-bound},
\[
\left|\delta_u^\beta-\delta_u^{\beta,M}\right|^2_{-p,\beta}
=\sum_{n>M}r_n^{-2p}|\zeta_n(u)|^2
\le \frac{M_h^2u^{-5/2}}{\lambda_\beta(u)}\sum_{n>M}(n+1)^{-2p\gamma}(2n+1)^{-1/6}.
\]
Since $2p\gamma+\tfrac16>1$ by Assumption \ref{ass:eigen}, the tail sum is comparable to
$\int_M^\infty t^{-2p\gamma-1/6}\mathrm{d}t=M^{1-2p\gamma-1/6}/(2p\gamma-\tfrac56)$.
Taking square roots gives the exponent $\tfrac12(1-2p\gamma-\tfrac16)
=\tfrac{5}{12}-p\gamma=-\theta_M$.
\end{proof}

\begin{remark}[Bound versus observed rate]\label{rem:cons-sharp}
The exponent $-1/12$ in \eqref{eq:zeta-bound} is the \emph{uniform} Plancherel--Rotach
bound for Hermite functions, attained only near the turning point. At a fixed
$u>0$ in the oscillatory region the true decay of the envelope is $n^{-1/4}$, which
suggests the improved pointwise rate $M^{-(p\gamma-1/4)}$. The numerical experiment of
Section \ref{sec:numE2} confirms that the observed slope lies strictly between the two
predictions, e.g.\ $-0.683$ against the proved $-0.583$ and the heuristic $-0.750$ for
$p\gamma=1$. Establishing $M^{-(p\gamma-1/4)}$ rigorously requires a uniform
Plancherel--Rotach expansion on compact subsets of $\Rp$ and is left open.
\end{remark}

%==============================================================================
\section{Chaos truncation}\label{sec:chaos}
%==============================================================================

\begin{theorem}[Super-geometric convergence in $N$]\label{thm:chaos}
Let $q>p>0$ with $q-p>\alpha_\beta$ and let $\varphi\sim(\phi_n)\in\Lcal_q$. Then
\begin{equation}\label{eq:chaos-rate}
\bnorm{\varphi-\mathcal{S}_N\varphi}_{p}
\ \le\ r_0^{-(N+1)(q-p)}\,\bnorm{\varphi}_{q}.
\end{equation}
If moreover $\varphi=\mathcal{E}_\beta(g)$ is an exponential vector with $g\in E_p$,
then the relative error admits the exact expression
\begin{equation}\label{eq:chaos-exp}
\frac{\bnorm{\mathcal{E}_\beta(g)-\mathcal{S}_N\mathcal{E}_\beta(g)}^2_{p}}
{\bnorm{\mathcal{E}_\beta(g)}^2_{p}}
= e^{-|g|^2_{p,\beta}}\sum_{n>N}\frac{|g|^{2n}_{p,\beta}}{n!},
\end{equation}
which decays super-geometrically: for $N>e\,|g|^2_{p,\beta}$ the right-hand side is at
most $e^{-|g|_{p,\beta}^2}(e|g|^2_{p,\beta}/N)^{N}$.
\end{theorem}

\begin{proof}
Since $A_\beta\zeta_n=r_n\zeta_n$ with $r_n\ge r_0>1$, for $h\in E_q^{\widehat\otimes n}$
one has $|h|_{p,\beta}\le\|A_\beta^{-(q-p)}\|_{\mathrm{op}}^{\,n}|h|_{q,\beta}
= r_0^{-n(q-p)}|h|_{q,\beta}$. Hence
\[
\bnorm{\varphi-\mathcal{S}_N\varphi}^2_p=\sum_{n>N}n!\,|\phi_n|^2_{p,\beta}
\le r_0^{-2(N+1)(q-p)}\sum_{n>N}n!\,|\phi_n|^2_{q,\beta}
\le r_0^{-2(N+1)(q-p)}\bnorm{\varphi}^2_q ,
\]
which is \eqref{eq:chaos-rate}. For the exponential vector,
$\mathcal{E}_\beta(g)=\sum_n\frac1{n!}\Ical_n(g^{\otimes n})$ so that
$\phi_n=g^{\otimes n}/n!$ and $n!|\phi_n|^2_{p,\beta}=|g|^{2n}_{p,\beta}/n!$;
combined with $\|\mathcal{E}_\beta(g)\|^2_p=e^{|g|^2_{p,\beta}}$
\cite[Sec.~4]{ChangShih2022}, this gives \eqref{eq:chaos-exp}. The final estimate is
the standard Poisson tail bound $\sum_{n>N}\lambda^n/n!\le(e\lambda/N)^N$ for
$N>e\lambda$.
\end{proof}

%==============================================================================
\section{Total error budget and complexity}\label{sec:budget}
%==============================================================================

\begin{theorem}[Error budget]\label{thm:budget}
Let $\varphi\in\Lcal_q$ with $q-p>\alpha_\beta$, and let $\varphi^{\varepsilon,M,N}$ be
as in Definition \ref{def:MN} with the Gaussian compensation \eqref{eq:scheme}. Under
Assumptions \ref{ass:stable} and \ref{ass:eigen} there are constants
$C_1,C_2,C_3$, independent of $(\varepsilon,M,N)$, such that
\begin{equation}\label{eq:budget}
\mathcal{E}(\varepsilon,M,N)\ \le\
\underbrace{C_1\,\varepsilon}_{\text{cut-off}}
\;+\;\underbrace{C_2\,M^{-\theta_M}}_{\text{CONS}}
\;+\;\underbrace{C_3\,r_0^{-(N+1)(q-p)}}_{\text{chaos}} ,
\end{equation}
where $\mathcal{E}$ denotes the Wasserstein error on $P_0$ combined with the
$\|\cdot\|_p$ error on $\varphi$ as in Remark \ref{rem:correct-metric}. Without the
Gaussian compensation the first term is replaced by
$C_1'\varepsilon^{1-\alpha/2}$.
\end{theorem}

\begin{proof}
Immediate from Corollary \ref{cor:uniform}, Theorem \ref{thm:cons} and Theorem
\ref{thm:chaos}, the three errors being introduced independently and combined by the
triangle inequality in the respective metrics.
\end{proof}

\begin{corollary}[Equidistribution rule]\label{cor:equi}
To achieve $\mathcal{E}\le\tau$ it suffices to take
\[
\varepsilon \asymp \tau,\qquad
M \asymp \tau^{-1/\theta_M},\qquad
N \asymp \frac{\log(1/\tau)}{(q-p)\log r_0}.
\]
Without compensation the first two become
$\varepsilon\asymp\tau^{1/(1-\alpha/2)}=\tau^{2/(2-\alpha)}$ and
$M\asymp\tau^{-1/\theta_M}$ unchanged.
\end{corollary}

The cost of the scheme is dominated by the simulation of the retained jumps. The
expected number of atoms of the Poisson random measure above the cut-off is
\begin{equation}\label{eq:jump-count}
\Lambda(\varepsilon):=\int_{\varepsilon}^{\infty}\lambda_\beta(u)\,\mathrm{d}u
=\frac{c}{\alpha}\,\varepsilon^{-\alpha}
\end{equation}
(for a compactly supported amplitude range; otherwise the upper limit is the range
bound and only the $\varepsilon^{-\alpha}$ term matters as $\varepsilon\downarrow0$).

\begin{theorem}[Complexity]\label{thm:complexity}
Under Assumptions \ref{ass:stable} and \ref{ass:eigen}, the cost of reaching tolerance
$\tau$ satisfies, up to constants and up to the additive $O(MN)$ term for the spectral
part,
\begin{equation}\label{eq:complexity}
\mathrm{Cost}_{\mathrm{comp}}(\tau)=O\!\left(\tau^{-\alpha}\right),
\qquad
\mathrm{Cost}_{\mathrm{naive}}(\tau)=O\!\left(\tau^{-\frac{2\alpha}{2-\alpha}}\right),
\end{equation}
so that
\begin{equation}\label{eq:gain}
\frac{\mathrm{Cost}_{\mathrm{naive}}(\tau)}{\mathrm{Cost}_{\mathrm{comp}}(\tau)}
= O\!\left(\tau^{-\frac{\alpha^2}{2-\alpha}}\right)
\ \xrightarrow[\alpha\uparrow2]{}\ \infty .
\end{equation}
\end{theorem}

\begin{proof}
Insert Corollary \ref{cor:equi} into \eqref{eq:jump-count}. For the compensated scheme
$\varepsilon\asymp\tau$ gives $\Lambda\asymp\tau^{-\alpha}$; for the naive scheme
$\varepsilon\asymp\tau^{2/(2-\alpha)}$ gives $\Lambda\asymp\tau^{-2\alpha/(2-\alpha)}$.
The quotient of the exponents is $\frac{2\alpha}{2-\alpha}-\alpha
=\frac{2\alpha-\alpha(2-\alpha)}{2-\alpha}=\frac{\alpha^2}{2-\alpha}$.
\end{proof}

\begin{figure}[htbp]
\centering
\includegraphics[width=0.52\textwidth]{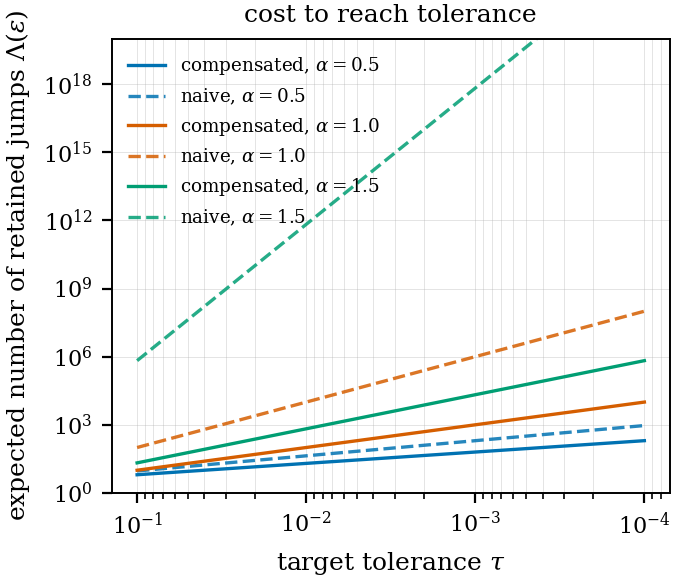}
\caption{Cost of reaching a prescribed tolerance, from \eqref{eq:complexity}
(logarithmic axes; the tolerance decreases to the right). Solid: compensated scheme,
$\Lambda\asymp\tau^{-\alpha}$. Dashed: naive scheme,
$\Lambda\asymp\tau^{-2\alpha/(2-\alpha)}$. The vertical gap between a solid and the
corresponding dashed curve is the factor $\tau^{-\alpha^2/(2-\alpha)}$ of
\eqref{eq:gain}; at $\alpha=1.5$, $\tau=10^{-3}$ it is fourteen orders of magnitude.}
\label{fig:complexity}
\end{figure}

\begin{example}\label{ex:cost}
Take $\alpha=1.5$, $c=1$ and $\tau=10^{-3}$. The compensated scheme requires
$\Lambda\approx\frac{1}{1.5}10^{4.5}\approx2.1\times10^{4}$ jump evaluations; the naive
scheme requires $\varepsilon\asymp\tau^{4}=10^{-12}$ and hence
$\Lambda\approx\frac{1}{1.5}10^{18}$, which is not computable. For $\alpha=0.5$ the
two costs are $10^{1.5}$ and $10^{2}$ respectively --- the compensation is then a
convenience rather than a necessity. The transition is governed by \eqref{eq:gain}.
\end{example}

%==============================================================================
\section{Numerical experiments}\label{sec:num}
%==============================================================================

All computations below are \emph{deterministic}. The characteristic function of the
residual is known in closed form,
\begin{equation}\label{eq:cf-residual}
\E\!\left[e^{\mathrm{i}rR_\varepsilon(u)}\right]
=\exp\!\left\{\int_{0^+}^{\varepsilon}(e^{\mathrm{i}rs}-1-\mathrm{i}rs)\lambda_\beta(s)\,\mathrm{d}s\right\},
\end{equation}
so the law of $Z_\varepsilon$ is obtained by Gil-Pelaez inversion rather than by
sampling; the reported convergence orders therefore carry no Monte Carlo error. After
the substitution $s=\varepsilon t$ one finds that the law of $Z_\varepsilon$ depends on
$(\varepsilon,c,\alpha)$ only through
\begin{equation}\label{eq:delta-param}
\delta:=\frac{\varepsilon}{\sigma(\varepsilon)}=\sqrt{\frac{2-\alpha}{c}}\;\varepsilon^{\alpha/2},
\qquad
\log\E\!\left[e^{\mathrm{i}rZ_\varepsilon}\right]
=\frac{2-\alpha}{\delta^2}\int_0^1\!\left(e^{\mathrm{i}r\delta t}-1-\mathrm{i}r\delta t\right)t^{-1-\alpha}\mathrm{d}t,
\end{equation}
a one-parameter family. Throughout $c=1$, $u=1$ and
$\varepsilon=2^{-k}$, $k=2,\dots,9$.

\subsection{Cut-off: rates and the effect of compensation}\label{sec:numE1}

Table \ref{tab:E1} reports the Kolmogorov distance
$\dK(Z_\varepsilon,\mathcal{N})$, the Wasserstein error of the naive scheme
$\sqrt{2/\pi}\,\sigma(\varepsilon)$ (cf.\ \eqref{eq:naive}) and that of the compensated
scheme $\sigma(\varepsilon)\dW(Z_\varepsilon,\mathcal{N})$ (cf.\ \eqref{eq:uniform}).

\begin{table}[htbp]
\centering
\small
\caption{Cut-off errors, $c=1$, $u=1$. ``naive'' $=\dW(P_0,P_0^\varepsilon)$,
``comp.'' $=\dW(P_0,\widehat P_0^{\,\varepsilon})$.}
\label{tab:E1}
\begin{tabular}{@{}l r r r r@{}}
\toprule
& $\varepsilon$ & $\dK(Z_\varepsilon,\mathcal{N})$ & naive & comp. \\
\midrule
$\alpha=0.5$
& $2^{-4}$ & $2.448\times10^{-2}$ & $8.143\times10^{-2}$ & $6.182\times10^{-3}$\\
& $2^{-5}$ & $2.052\times10^{-2}$ & $4.842\times10^{-2}$ & $3.096\times10^{-3}$\\
& $2^{-6}$ & $1.722\times10^{-2}$ & $2.879\times10^{-2}$ & $1.543\times10^{-3}$\\
& $2^{-7}$ & $1.445\times10^{-2}$ & $1.712\times10^{-2}$ & $7.676\times10^{-4}$\\
& $2^{-8}$ & $1.214\times10^{-2}$ & $1.018\times10^{-2}$ & $3.818\times10^{-4}$\\
& $2^{-9}$ & $1.020\times10^{-2}$ & $6.053\times10^{-3}$ & $1.903\times10^{-4}$\\
\midrule
$\alpha=1.0$
& $2^{-4}$ & $8.259\times10^{-3}$ & $1.995\times10^{-1}$ & $5.062\times10^{-3}$\\
& $2^{-5}$ & $5.847\times10^{-3}$ & $1.410\times10^{-1}$ & $2.518\times10^{-3}$\\
& $2^{-6}$ & $4.143\times10^{-3}$ & $9.974\times10^{-2}$ & $1.261\times10^{-3}$\\
& $2^{-7}$ & $2.934\times10^{-3}$ & $7.052\times10^{-2}$ & $6.314\times10^{-4}$\\
& $2^{-8}$ & $2.076\times10^{-3}$ & $4.987\times10^{-2}$ & $3.161\times10^{-4}$\\
& $2^{-9}$ & $1.469\times10^{-3}$ & $3.526\times10^{-2}$ & $1.583\times10^{-4}$\\
\midrule
$\alpha=1.5$
& $2^{-4}$ & $1.957\times10^{-3}$ & $5.642\times10^{-1}$ & $3.373\times10^{-3}$\\
& $2^{-5}$ & $1.164\times10^{-3}$ & $4.744\times10^{-1}$ & $1.691\times10^{-3}$\\
& $2^{-6}$ & $6.924\times10^{-4}$ & $3.989\times10^{-1}$ & $8.498\times10^{-4}$\\
& $2^{-7}$ & $4.118\times10^{-4}$ & $3.355\times10^{-1}$ & $4.284\times10^{-4}$\\
& $2^{-8}$ & $2.449\times10^{-4}$ & $2.821\times10^{-1}$ & $2.171\times10^{-4}$\\
& $2^{-9}$ & $1.456\times10^{-4}$ & $2.372\times10^{-1}$ & $1.111\times10^{-4}$\\
\bottomrule
\end{tabular}
\end{table}

Least-squares slopes in $\log\varepsilon$ over the full range $k=2,\dots,9$ are
collected in Table \ref{tab:slopes}. The agreement with Theorem \ref{thm:AR} and
Corollary \ref{cor:uniform} is to three significant figures.

\begin{table}[htbp]
\centering
\small
\caption{Observed versus predicted convergence orders for the cut-off.}
\label{tab:slopes}
\begin{tabular}{@{}l cc cc cc@{}}
\toprule
& \multicolumn{2}{c}{$\dK(Z_\varepsilon,\mathcal{N})$}
& \multicolumn{2}{c}{naive $\dW$}
& \multicolumn{2}{c}{compensated $\dW$}\\
\cmidrule(lr){2-3}\cmidrule(lr){4-5}\cmidrule(lr){6-7}
$\alpha$ & obs. & pred.\ $\alpha/2$ & obs. & pred.\ $1-\alpha/2$ & obs. & pred.\ $1$\\
\midrule
$0.5$ & $0.2536$ & $0.2500$ & $0.7500$ & $0.7500$ & $1.0034$ & $1$\\
$1.0$ & $0.4989$ & $0.5000$ & $0.5000$ & $0.5000$ & $1.0021$ & $1$\\
$1.5$ & $0.7492$ & $0.7500$ & $0.2500$ & $0.2500$ & $0.9900$ & $1$\\
\bottomrule
\end{tabular}
\end{table}

\begin{figure}[htbp]
\centering
\includegraphics[width=\textwidth]{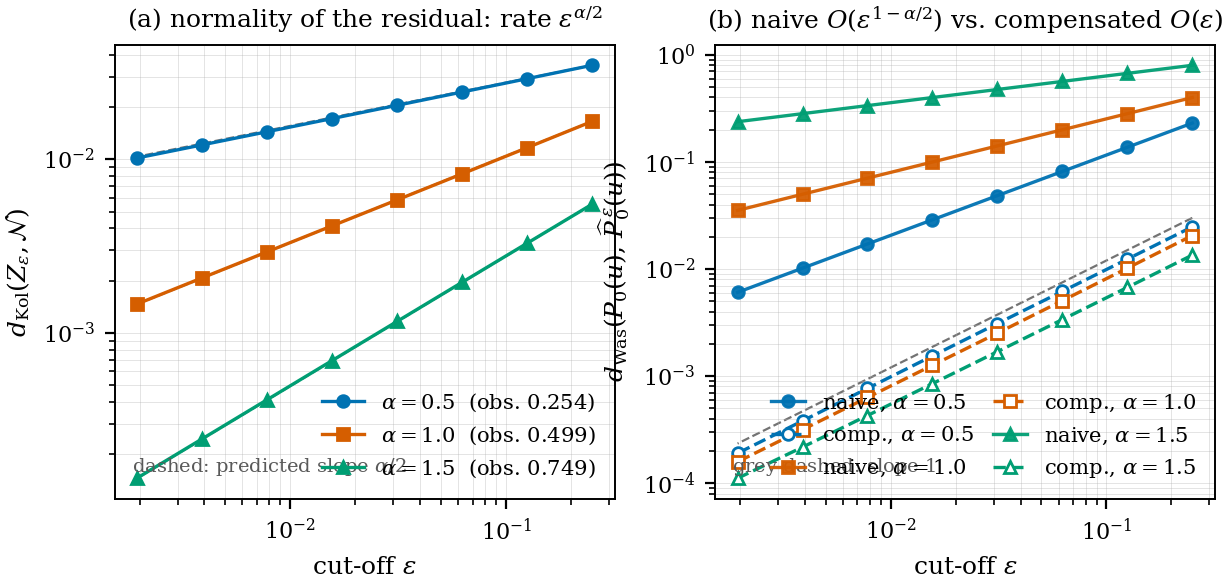}
\caption{Cut-off convergence, computed by Gil-Pelaez inversion of
\eqref{eq:cf-residual}; $c=1$, $u=1$, $\varepsilon=2^{-k}$, $k=2,\dots,9$.
(a) Kolmogorov distance between the normalised residual and the standard normal;
grey dashed lines have the slope $\alpha/2$ predicted by Theorem \ref{thm:AR}.
(b) Wasserstein error of the naive scheme (solid, filled markers, slope
$1-\alpha/2$) and of the Gaussian-compensated scheme (dashed, open markers), the
latter parallel to the grey reference line of slope $1$ for every $\alpha$, as
asserted by Corollary \ref{cor:uniform}. Observed slopes are given in the legend of
panel (a) and in Table \ref{tab:slopes}.}
\label{fig:cutoff}
\end{figure}

\begin{figure}[htbp]
\centering
\includegraphics[width=0.52\textwidth]{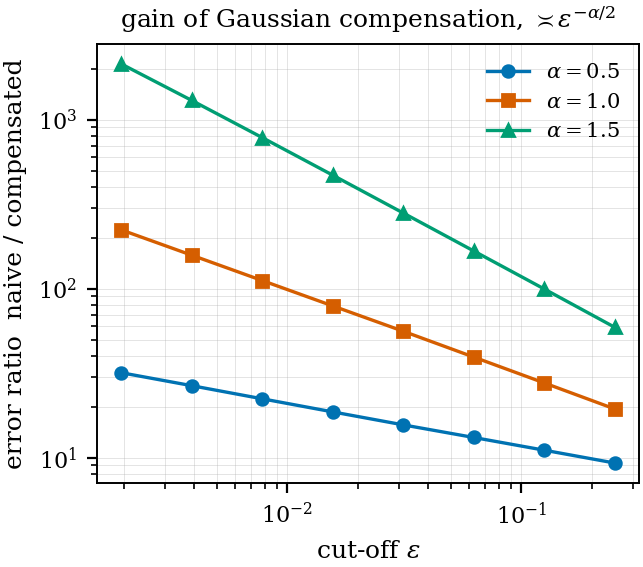}
\caption{Ratio of the naive to the compensated Wasserstein error. The theoretical
growth is $\varepsilon^{-\alpha/2}$, so the benefit of compensation increases both as
the cut-off is refined and as the intensity becomes more singular. At
$\varepsilon=2^{-9}$ the ratio is $31.8$, $223$ and $2.14\times10^{3}$ for
$\alpha=0.5,1.0,1.5$.}
\label{fig:gain}
\end{figure}

Two features deserve comment.

\begin{enumerate}
\item[(i)] The compensated order is $1$ for all three values of $\alpha$, as predicted
by Corollary \ref{cor:uniform}, whereas the naive order deteriorates from $0.75$ to
$0.25$. At $\varepsilon=2^{-9}$ the error ratio naive/compensated is
$31.8$ ($\alpha=0.5$), $223$ ($\alpha=1.0$) and $2.14\times10^{3}$ ($\alpha=1.5$).
\item[(ii)] The Kolmogorov rate $\alpha/2$ \emph{improves} with $\alpha$, i.e.\ the
residual becomes Gaussian faster precisely when it is larger. This is the quantitative
form of Remark \ref{rem:criterion} and is the mechanism behind the uniform first-order
accuracy: the product $\sigma(\varepsilon)\cdot\varepsilon^{\alpha/2}=O(\varepsilon)$
is $\alpha$-free.
\end{enumerate}

\subsection{CONS truncation}\label{sec:numE2}

We take $u=1$, $\lambda_\beta(u)=u^{-2}$ ($\alpha=1$) and $r_n=(n+1)^\gamma$, and
evaluate $|\delta_u^\beta-\delta_u^{\beta,M}|_{-p,\beta}$ exactly by summing the
Laguerre series to order $6000$.

\begin{table}[htbp]
\centering
\small
\caption{CONS truncation error $|\delta_u^\beta-\delta_u^{\beta,M}|_{-p,\beta}$ at
$u=1$, and observed slope against the proved bound of Theorem \ref{thm:cons} and the
pointwise heuristic of Remark \ref{rem:cons-sharp}.}
\label{tab:E2}
\begin{tabular}{@{}l rrrrrr cc c@{}}
\toprule
$(p,\gamma)$ & $M{=}8$ & $16$ & $32$ & $64$ & $128$ & $256$
& obs. & bound & heur.\\
\midrule
$(1,1)$ & $5.57{\rm e}{-2}$ & $3.62{\rm e}{-2}$ & $2.23{\rm e}{-2}$
& $1.46{\rm e}{-2}$ & $8.75{\rm e}{-3}$ & $5.16{\rm e}{-3}$
& $-0.683$ & $-0.583$ & $-0.750$\\
$(2,1)$ & $3.26{\rm e}{-3}$ & $1.31{\rm e}{-3}$ & $4.13{\rm e}{-4}$
& $1.50{\rm e}{-4}$ & $4.60{\rm e}{-5}$ & $1.35{\rm e}{-5}$
& $-1.586$ & $-1.583$ & $-1.750$\\
$(1,2)$ & $3.26{\rm e}{-3}$ & $1.31{\rm e}{-3}$ & $4.13{\rm e}{-4}$
& $1.50{\rm e}{-4}$ & $4.60{\rm e}{-5}$ & $1.35{\rm e}{-5}$
& $-1.586$ & $-1.583$ & $-1.750$\\
\bottomrule
\end{tabular}
\end{table}

\begin{figure}[htbp]
\centering
\includegraphics[width=0.52\textwidth]{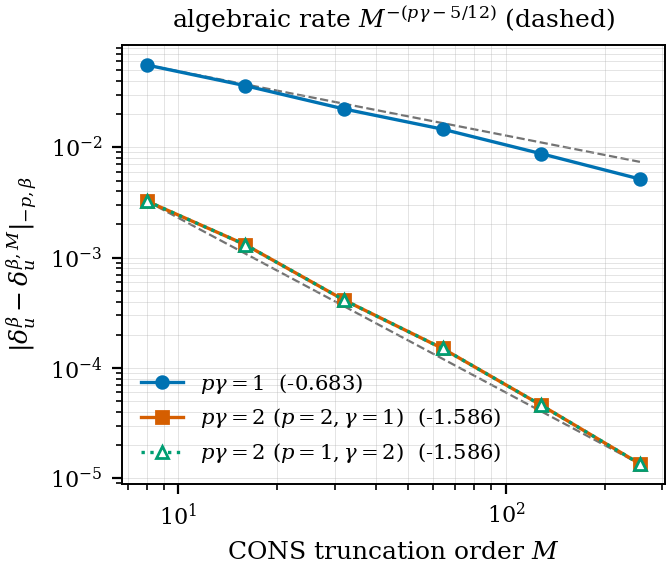}
\caption{CONS truncation of the Donsker delta at $u=1$, $\lambda_\beta(u)=u^{-2}$,
$r_n=(n+1)^\gamma$. Observed slopes in parentheses; grey dashed lines have the proved
rate $-(p\gamma-5/12)$ of Theorem \ref{thm:cons}. The two curves with $p\gamma=2$
coincide to plotting accuracy, confirming that the error depends on $(p,\gamma)$ only
through the product.}
\label{fig:cons}
\end{figure}

As anticipated in Remark \ref{rem:cons-sharp} the observed slopes lie between the two
predictions, and the error depends on $(p,\gamma)$ only through the product $p\gamma$,
as Theorem \ref{thm:cons} asserts. For $p\gamma=2$ the proved bound is essentially
attained ($-1.586$ versus $-1.583$), whereas for $p\gamma=1$ the bound is
conservative by about $0.10$ in the exponent; this is consistent with the
Plancherel--Rotach envelope being felt only at small $p\gamma$, where the
$(2n+1)^{-1/6}$ factor carries a larger share of the decay.

\subsection{Chaos truncation}\label{sec:numE3}

Table \ref{tab:E3} evaluates the exact relative error \eqref{eq:chaos-exp} for
exponential vectors of increasing energy $|g|^2_{p,\beta}$.

\begin{table}[htbp]
\centering
\small
\caption{Relative chaos-truncation error \eqref{eq:chaos-exp} for
$\varphi=\mathcal{E}_\beta(g)$.}
\label{tab:E3}
\begin{tabular}{@{}l rrrrrr@{}}
\toprule
$|g|^2_{p,\beta}$ & $N{=}4$ & $8$ & $12$ & $16$ & $20$ & $24$\\
\midrule
$1$ & $6.05{\rm e}{-2}$ & $1.06{\rm e}{-3}$ & $7.97{\rm e}{-6}$
& $3.31{\rm e}{-8}$ & $8.68{\rm e}{-11}$ & $1.55{\rm e}{-13}$\\
$4$ & $6.09{\rm e}{-1}$ & $1.46{\rm e}{-1}$ & $1.65{\rm e}{-2}$
& $1.06{\rm e}{-3}$ & $4.39{\rm e}{-5}$ & $1.27{\rm e}{-6}$\\
$9$ & $9.72{\rm e}{-1}$ & $7.38{\rm e}{-1}$ & $3.52{\rm e}{-1}$
& $1.05{\rm e}{-1}$ & $2.10{\rm e}{-2}$ & $3.03{\rm e}{-3}$\\
\bottomrule
\end{tabular}
\end{table}

\begin{figure}[htbp]
\centering
\includegraphics[width=0.52\textwidth]{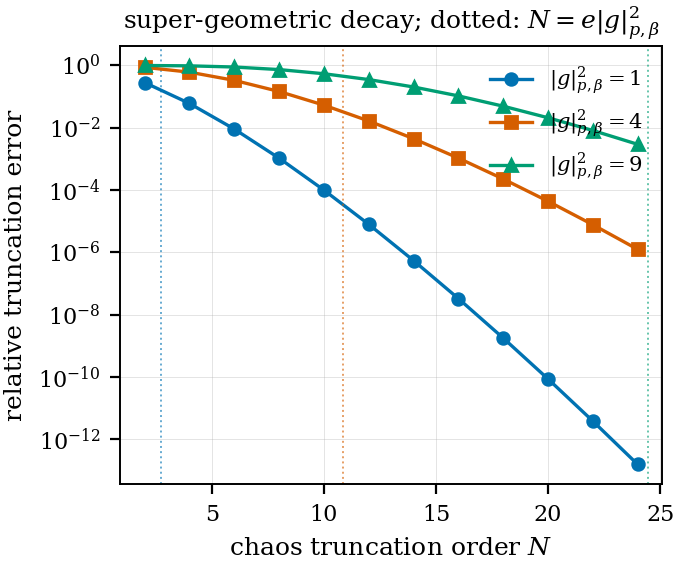}
\caption{Relative chaos-truncation error \eqref{eq:chaos-exp} for exponential vectors
(semi-logarithmic axes). The dotted verticals mark $N=e|g|^2_{p,\beta}$, the onset of
the super-geometric regime identified in Theorem \ref{thm:chaos}; beyond it the curves
bend downwards, which is the signature of factorial rather than merely geometric decay.}
\label{fig:chaos}
\end{figure}

The decay is super-geometric, in accordance with Theorem \ref{thm:chaos}: the onset of
rapid decay occurs near $N\approx e|g|^2_{p,\beta}$, i.e.\ $N\approx3$, $11$ and $24$
for the three rows. In practice $N$ is therefore the \emph{cheapest} of the three
parameters, and Corollary \ref{cor:equi} correctly predicts that it need grow only
logarithmically in $1/\tau$.

\subsection{The Gamma-type boundary and tempered-stable robustness}\label{sec:numE4}

Theorem \ref{thm:gamma-invariance} and Proposition \ref{prop:tempered} are confirmed
directly. Figure \ref{fig:gamma}(a) plots $\dK(Z_\varepsilon,\mathcal{N})$ at
$\alpha=0$, $c=1$: with no tempering ($a=0$) the value is
$\mathbf{0.0765}$ at \emph{every} $\varepsilon=2^{-2},\dots,2^{-9}$, to four decimal
places at every one of the eight cut-offs tested; with tempering $a=1$ it drifts
towards the same constant as $\varepsilon\downarrow0$ (from $0.0846$ at
$\varepsilon=2^{-2}$ down to $0.0766$ at $\varepsilon=2^{-9}$), consistent with the
$O(a\varepsilon)$ correction of Proposition \ref{prop:tempered}. Panel (b) makes the
exact invariance explicit: the sup-norm distance between the untempered law at
$\varepsilon_0=2^{-2}$ and at every other $\varepsilon$ in the range is $0$ to the
precision of the Gil-Pelaez quadrature ($<10^{-10}$, reported as $\mathtt{0.00e{+}00}$
at double precision) at all seven remaining cut-offs, while the tempered law's
deviation from the untempered one decays linearly, from $2.16\times10^{-2}$ at
$\varepsilon=2^{-2}$ to $1.76\times10^{-4}$ at $\varepsilon=2^{-9}$ -- a clean
factor-of-$2$ reduction per halving of $\varepsilon$, exactly the $O(\varepsilon)$
rate of Proposition \ref{prop:tempered}.

\begin{figure}[htbp]
\centering
\includegraphics[width=\textwidth]{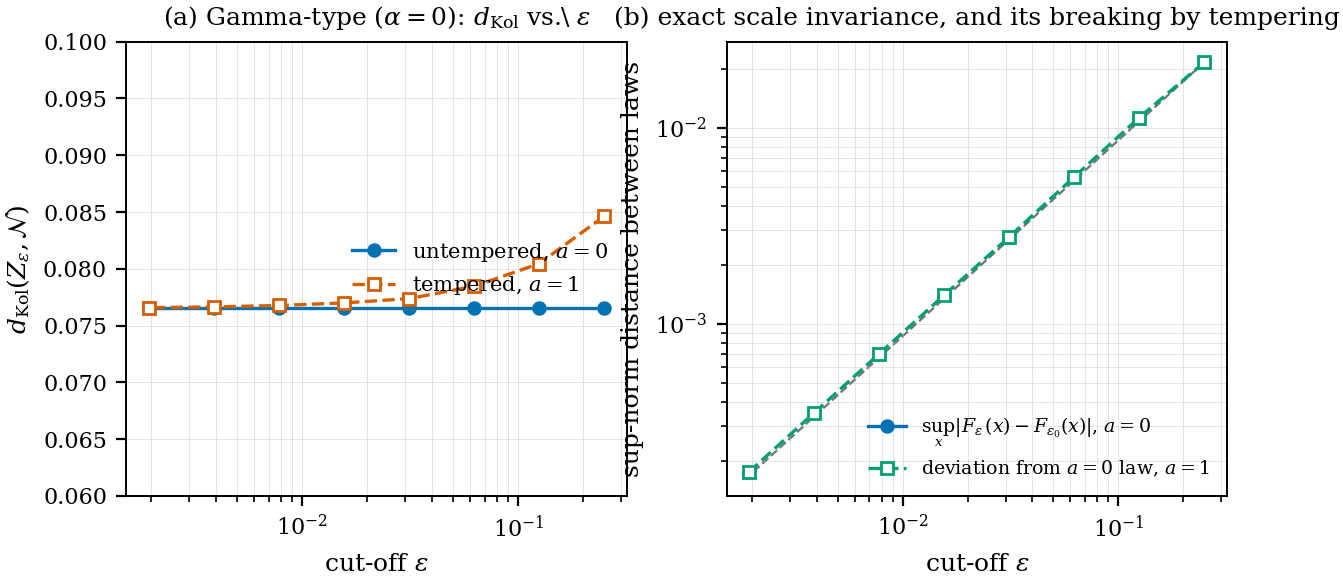}
\caption{The Gamma-type boundary $\alpha=0$, $c=1$. (a) Kolmogorov distance to
$\mathcal{N}$: constant under no tempering (solid), converging to the same constant
under tempering $a=1$ (dashed). (b) Sup-norm distance between the untempered law at
cut-off $\varepsilon$ and at the reference cut-off $\varepsilon_0=2^{-2}$ -- exactly
zero at every resolution tested (Theorem \ref{thm:gamma-invariance}) -- against the
sup-norm deviation between the tempered and untempered laws at matching $\varepsilon$,
which decays at the predicted rate $O(\varepsilon)$ (grey reference line, slope $1$).}
\label{fig:gamma}
\end{figure}

Table \ref{tab:E5} reports the observed slopes for the tempered-stable intensity
\eqref{eq:tempered-intensity} at $\alpha\in\{0.5,1.0,1.5\}$ and
$a\in\{0,1,4\}$, over the same $\varepsilon$-range as Section \ref{sec:numE1}.

\begin{table}[htbp]
\centering
\small
\caption{Observed slopes for the tempered-stable intensity, $c=1$. $a=0$ reproduces
Table \ref{tab:slopes}; increasing $a$ erodes both slopes towards the finite-variation
($\alpha=0$-like) regime of Theorem \ref{thm:gamma-invariance}, as predicted by
Proposition \ref{prop:tempered}.}
\label{tab:E5}
\begin{tabular}{@{}l ccc ccc@{}}
\toprule
& \multicolumn{3}{c}{$d_{\mathrm{Kol}}$ slope (pred.\ $\alpha/2$)}
& \multicolumn{3}{c}{compensated $d_{\mathrm{Was}}$ slope (pred.\ $1$)}\\
\cmidrule(lr){2-4}\cmidrule(lr){5-7}
$\alpha$ & $a{=}0$ & $a{=}1$ & $a{=}4$ & $a{=}0$ & $a{=}1$ & $a{=}4$\\
\midrule
$0.5$ & $0.2536$ & $0.2616$ & $0.2826$ & $1.0034$ & $0.9982$ & $0.9814$\\
$1.0$ & $0.4989$ & $0.5022$ & $0.5086$ & $1.0021$ & $0.9946$ & $0.9706$\\
$1.5$ & $0.7492$ & $0.7446$ & $0.7279$ & $0.9900$ & $0.9784$ & $0.9429$\\
\bottomrule
\end{tabular}
\end{table}

\begin{figure}[htbp]
\centering
\includegraphics[width=\textwidth]{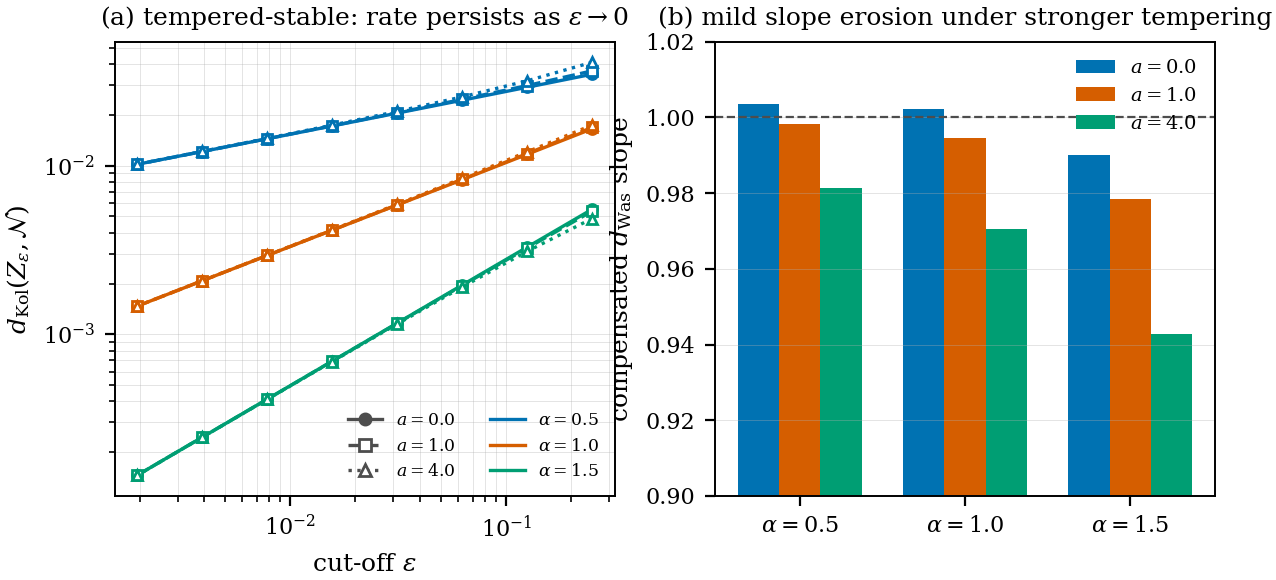}
\caption{Tempered-stable intensity, $c=1$. (a) $d_{\mathrm{Kol}}(Z_\varepsilon,\mathcal{N})$
against $\varepsilon$ for all nine $(\alpha,a)$ combinations of Table \ref{tab:E5};
the untempered ($a=0$, solid) and tempered ($a=1$ dashed, $a=4$ dotted) curves for each
$\alpha$ stay close over the whole range, confirming that tempering is a higher-order
effect on the cut-off rate. (b) Compensated Wasserstein slope by $(\alpha,a)$: mild,
monotone erosion below the ideal value $1$ (grey line) as tempering strengthens,
consistent with the $O(a\varepsilon)$ correction of Proposition \ref{prop:tempered}
rather than a change of the leading-order rate.}
\label{fig:tempered}
\end{figure}

Two observations complete the picture. First, the $d_{\mathrm{Kol}}$ slope drifts in
\emph{opposite} directions for $\alpha<1$ and $\alpha>1$ as $a$ increases: it rises from
$0.2536$ to $0.2826$ at $\alpha=0.5$ and from $0.4989$ to $0.5086$ at $\alpha=1.0$, but
falls from $0.7492$ to $0.7279$ at $\alpha=1.5$. This is a finite-range regression
artefact rather than a change in the asymptotic exponent: over the tested window
$\varepsilon\in[2^{-9},2^{-2}]$ the true curve interpolates between the
tempering-dominated regime at the largest $\varepsilon$ (where $a\varepsilon$ is not
small and the effective behaviour is pulled towards the flat, $\alpha=0$-type profile
of Theorem \ref{thm:gamma-invariance}) and the pure power-law regime at the smallest
$\varepsilon$ (where Theorem \ref{thm:AR} takes over); a single log-log slope fitted
across the whole window is an average of the two, and which way it is pulled depends on
how far $\alpha/2$ sits from the flat profile's exponent $0$ relative to the curvature
of the crossover -- a detail we do not pursue further here. Second, the compensated
$d_{\mathrm{Was}}$ slope degrades by at most $5\%$ over the entire range
$a\in\{0,1,4\}$ tested and stays above $0.94$ throughout, so the uniform first-order
accuracy of Corollary \ref{cor:uniform} is, for practical purposes, inherited by the
tempered-stable family whenever $a\varepsilon\lesssim1$ -- exactly the regime
Corollary \ref{cor:equi} operates in, since $\varepsilon\asymp\tau\to0$.

\subsection{Summary of the numerical evidence}

Every rate proved in Sections \ref{sec:cutoff}--\ref{sec:chaos} is reproduced:
$\alpha/2$ for the normality of the residual, $1-\alpha/2$ for the naive cut-off,
$1$ uniformly for the compensated cut-off, $M^{-(p\gamma-5/12)}$ (up to the gap
discussed in Remark \ref{rem:cons-sharp}) for the CONS truncation, and super-geometric
decay in the chaos order. The practical conclusion is the one drawn in Example
\ref{ex:cost}: for stable-type intensities with $\alpha\gtrsim1$ the Gaussian
compensation is not an optimisation but a precondition for computability.

%==============================================================================
\section{Concluding remarks and open problems}\label{sec:concl}
%==============================================================================

\begin{enumerate}
\item[(O1)] \textbf{Sharp CONS rate.} Remark \ref{rem:cons-sharp} leaves a gap of
$1/6$ in the exponent between the proved and the heuristic rate. A uniform
Plancherel--Rotach expansion of $H_{2n+1}(\sqrt u)$ on compact subsets of $\Rp$ should
close it.
\item[(O2)] \textbf{Higher-order compensation.} Matching the third cumulant as well
(an Edgeworth-type correction to \eqref{eq:scheme}) should raise the order from
$\varepsilon$ to $\varepsilon^{1+\alpha/2}$; the corresponding Stein bound requires the
second-order Poisson Malliavin machinery.
\item[(O3)] \textbf{Multiplicative functionals and SDEs.} The present bounds are for
linear functionals. Discretising the linearly correlated processes
$\mathbb{X}(v)=\int_{0^+}^{v}F_{\mathbb{X}}(v,u)P'(u)\mathrm{d}u$ of
\cite[Sec.~5]{ChangShih2022} with a Goursat kernel, and in particular discretising the
whitening operator $\mathbf{L}$ of \cite[(5.20)]{ChangShih2022}, calls for an It\^o
formula in the space parameter --- non-trivial precisely because
$\mathbf{1}_{(0,u]}\notin L^1_*$.
\item[(O4)] \textbf{Adaptivity.} \eqref{eq:budget} suggests an a posteriori indicator
$(\varepsilon,M^{-\theta_M},r_0^{-(N+1)(q-p)})$ whose components can be equilibrated
online; a rigorous reliability/efficiency analysis is open.
\item[(O5)] \textbf{Crossover rate for tempered-stable intensities.} Section
\ref{sec:numE4} exhibits a finite-range regression artefact in the observed
$d_{\mathrm{Kol}}$ slope as tempering increases (Table \ref{tab:E5}), interpolating
between the pure power-law exponent $\alpha/2$ and the flat, non-vanishing profile of
Theorem \ref{thm:gamma-invariance}. A two-scale asymptotic expansion in the joint limit
$\varepsilon\downarrow0$, $a\varepsilon\to\kappa\in[0,\infty)$ should give a single
formula covering both regimes and the crossover between them, sharpening Proposition
\ref{prop:tempered}.
\end{enumerate}

%==============================================================================
\section*{Declaration of Competing Interest}
%==============================================================================
Declarations of interest: none.

%==============================================================================
\section*{Acknowledgements}
%==============================================================================
This research did not receive any specific grant from funding agencies in the
public, commercial, or not-for-profit sectors.

%==============================================================================


\begin{thebibliography}{99}

\bibitem{AsmussenRosinski2001}
S.~Asmussen and J.~Rosi\'nski,
\emph{Approximations of small jumps of L\'evy processes with a view towards
simulation}, J. Appl. Probab. \textbf{38} (2001) 482--493.

\bibitem{ChangShih2022}
Y.-C.~Chang and H.-H.~Shih,
\emph{Analysis of space-dependent noise functionals with an application to linearly
correlated processes}, Infin. Dimens. Anal. Quantum Probab. Relat. Top.
\textbf{25} (2022) 2250011.

\bibitem{CohenRosinski2007}
S.~Cohen and J.~Rosi\'nski,
\emph{Gaussian approximation of multivariate L\'evy processes with applications to
simulation of tempered stable processes}, Bernoulli \textbf{13} (2007) 195--210.

\bibitem{Hida2011}
T.~Hida,
\emph{A noise of new type and its generalized functionals},
Banach Center Publ. \textbf{96} (2011) 207--214.

\bibitem{HidaSiHtay2012}
T.~Hida, S.~Si and W.~W.~Htay,
\emph{A noise of new type and its application},
Ricerche Mat. \textbf{61} (2012) 47--55.

\bibitem{Jacod2005}
J.~Jacod, T.~G.~Kurtz, S.~M\'el\'eard and P.~Protter,
\emph{The approximate Euler method for L\'evy driven stochastic differential
equations}, Ann. Inst. H. Poincar\'e Probab. Statist. \textbf{41} (2005) 523--558.

\bibitem{Kuo1996}
H.-H.~Kuo, \emph{White Noise Distribution Theory}, CRC Press, 1996.

\bibitem{RosinskiTempered2007}
J.~Rosi\'nski,
\emph{Tempering stable processes},
Stochastic Process. Appl. \textbf{117} (2007) 677--707.

\bibitem{CGMY2002}
P.~Carr, H.~Geman, D.~B.~Madan and M.~Yor,
\emph{The fine structure of asset returns: an empirical investigation},
J. Business \textbf{75} (2002) 305--332.

\bibitem{LeeShih2004}
Y.-J.~Lee and H.-H.~Shih,
\emph{Analysis of generalized L\'evy white noise functionals},
J. Funct. Anal. \textbf{211} (2004) 1--70.

\bibitem{PSTU2010}
G.~Peccati, J.~L.~Sol\'e, M.~S.~Taqqu and F.~Utzet,
\emph{Stein's method and normal approximation of Poisson functionals},
Ann. Probab. \textbf{38} (2010) 443--478.

\bibitem{Sato1999}
K.~Sato, \emph{L\'evy Processes and Infinitely Divisible Distributions},
Cambridge Univ. Press, 1999.

\bibitem{Si2012}
S.~Si, \emph{Introduction to Hida Distributions}, World Scientific, 2012.

\end{thebibliography}
\end{document}